\documentclass{amsart}
\usepackage{amssymb}
\usepackage{graphicx}
\def\({\left(}
\def\){\right)}

\newtheorem{lema}{Lemma}[section]

\newtheorem*{teorema*}{Theorem}

\newtheorem{remark}[lema]{Remark}

\newtheorem{example}[lema]{Example}

\newtheorem{lemma}{Lemma}[section]

\newtheorem{theorem}[lema]{Theorem}

\newtheorem{proposition}[lema]{Proposition}

\newtheorem{definition}[lema]{Definition}

  {\par \hfill \fbox{}}
  {\par \hfill \fbox{}}
  {\par \hfill }
  {\par \hfill }

\def\beq{\begin{equation}}
\def\eeq{\end{equation}}

\def\epsilon{\varepsilon}

\begin{document}

\title{Arzela Ascoli theorem in quasi cone metric space}

\author{Dr. Shallu Sharma}
\address{Department of Mathematics, University of Jammu } 
\email{shallujamwal09@gmail.com}

\author{Iqbal Kour}
\address{Department of Mathematics, University of Jammu} \email{iqbalkour208@gmail.com}
\author{Sahil Billawria}
\address{Department of Mathematics,University Of Jammu, India}
\email{sahilbillawria2@gmail.com}


\date{}

\begin{abstract}
In this paper we investigate Arzela Ascoli Theorem in quasi cone metric space, which is a generalization of metric space. We prove some interesting results using forward and backward toplologies, forward and backward continuity and forward and backward totally boundedness. A new notion forward and backward equicontinuity is also introduced.   
\end{abstract}

\maketitle{ }

\section{Introduction}
A distance function is said to be quasi metric if the symmetry axiom is dropped in the definition of metric (see, for example \cite{A}, \cite{YHC}). Numerous authors have provided various definitions of quasi metric (see, for example \cite{AK}). Since quasi metric is a generalized form of metric space the  concept of quasi cone metric space broadens the notion of quasi metric. Many researches have been carried out particularly in the area of fixed point theory (see, for example \cite{SN}).
In this paper we have used the definition given by Abdeljawad and Karapinar \cite{AK}. We study topologies in a quasi cone metric space and define equicontinuity for a family of functions defined on quasi cone metric space. We also prove Arzela Ascoli theorem in quasi cone metric space.     
\section{Preliminaries}
\begin{definition}\cite{HZ}
Let $\mathsf{E}$ be a real Banach space. A set $\mathsf{P} \subset \mathsf{E}$ is called a cone if
\begin {itemize}
\item[(i)] $\bar{\mathsf{P}}=\mathsf{P},\mathsf{P} \neq \{0\}$ and $\mathsf{P} \neq \phi.$
\item[(ii)]  $\mathsf{x},\mathsf{y} \in \mathsf{P}$ and $s,t\in \mathsf{P}$ where  $s,t$ are non negative real numbers, then $s \mathsf{x}+t\mathsf{y}\in \mathsf{P}.$
\item[(iii)] $\mathsf{x}=0$ when $\mathsf{x} \in \mathsf{P}$ and $-\mathsf{x} \in \mathsf{P}.$
\end {itemize}
\end{definition}
For a cone $\mathsf{P}$ in $\mathsf{E},$ we define partial ordering $\leq$ with respect to $\mathsf{P}$  as follows: $\mathsf{x}\leq\mathsf{y}$ if and only if $\mathsf{y}-\mathsf{x}\in \mathsf{P}$ and $\mathsf{x} \ll \mathsf{y}$ if and only if $\mathsf{y}-\mathsf{x}\in {\mathsf{P}^{\circ}},$ $\mathsf{x}\leq\mathsf{y}$ indicates that $\mathsf{x}<\mathsf{y}$ and $\mathsf{x}\neq \mathsf{y}.$

\begin{definition}\cite{AK}
Let $\mathsf{X} \neq \phi$ be a set and $d:\mathsf{X}\times \mathsf{X} \to \mathsf{E}$ be a mapping that satisfies   
\begin{itemize}
\item[(i)] $d(\mathsf{x},\mathsf{y})\geq 0 ~\forall~ \mathsf{x},\mathsf{y}\in \mathsf{X}.$
\item[(ii)] $d(\mathsf{x},\mathsf{y})=0$ if and only if $\mathsf{x}=\mathsf{y}.$
\item[(iii)] $d(\mathsf{x},\mathsf{y})\leq d(\mathsf{x},\mathsf{z})+d(\mathsf{z},\mathsf{y})~\forall~ \mathsf{x},\mathsf{y},\mathsf{z}\in \mathsf{X}.$
\end{itemize}
Then the ordered pair $(\mathsf X,d)$ is called a quasi cone metric space.
\end{definition}
Further, we shall assume that $(\mathsf {X},d)$ is a quasi cone metric space, $\mathsf{E}$ is a real Banach space, $\mathsf{P}$ is a cone in $\mathsf{E}$ with non empty interior and $\leq$ denotes the partial ordering with respect to $\mathsf P.$
\begin{definition}\cite{YHC}
A sequence $\{\mathsf{x_{n}}\}$ is forward convergent (resp. backward convergent) to $\mathsf{x_{0}}\in \mathsf{X}$ if for each $u\in \mathsf P^{\circ},$ there is $\mathsf{N}\in \mathbb{N}$ such that $d(\mathsf{x_{0}},\mathsf{x_{n}})\ll u$ (resp. $d(\mathsf{x_{n}},\mathsf{x_{0}})\ll u$) for all  $n\geq \mathsf{N}.$
\end{definition}
\begin{definition}\cite{YHC}
If every sequence in $\mathsf{X}$ has a forward (or backward) convergent subsequence with limit in $\mathsf{A}$, then the set  $\mathsf{A}\subset X$ is forward (or backward) sequentially compact.
\end{definition}
\begin{definition}\cite{YHC}
Suppose $(\mathsf{X},d_{\mathsf{X}})$ and $(\mathsf{Y},d_{\mathsf{Y}})$ are two quasi cone metric spaces. A function $f$ from $\mathsf{X}$ to $\mathsf{Y}$ is said to $ff$-continuous at a point $x\in \mathsf{X}$ if $f(x_{n})$ is forward convergent to $f$ in $\mathsf{Y}$ when $x_{n}$ is forward convergent to $x$ in $\mathsf{X}.$ The space of all $ff-$continuous functions from $\mathsf{X}$ to $\mathsf{Y}$ is denoted by $C(\mathsf{X},\mathsf{Y}).$ 
\end{definition}
\begin{definition}\cite{YHC}
Let $(\mathsf{X},d_{\mathsf{X}})$ and $(\mathsf{Y},d_{\mathsf{Y}})$ be quasi cone metric spaces. The space of functions from $\mathsf{X}$ to $\mathsf{Y}$ is denoted by $\mathsf{Y}^{\mathsf{X}}.$ Let $f,g\in \mathsf{Y}^{\mathsf{X}}.$ Then the uniform metric associated with this space is defined as:
$$\rho(g,h)=\sup\{\bar{d}_{\mathsf{Y}}(g(\mathsf x),h(\mathsf x)):\mathsf x\in \mathsf{X}\}.$$ 
Here $\bar{d}_{\mathsf{Y}}(g(\mathsf x),h(\mathsf x))=\min \{{d}_{\mathsf{Y}}(g(\mathsf x),h(\mathsf x)),1\}. $
 \end{definition}
\section{Arzela Ascoli Theorem in quasi cone metric space}  

\begin{definition}
The forward open balls in quasi cone metric space $(\mathsf X,d)$ are defined by:
$$\mathsf B_{f}(\mathsf x,\mathsf u)=\{y\in \mathsf {X}:d(\mathsf x,\mathsf y)\ll u\},~\mbox {for}~ \mathsf x\in \mathsf X~\mbox{ and}~ u\in \mathsf{P}^{\circ}.$$\\ 
Similarly, the backward open balls are defined by:
$$\mathsf B_{b}(\mathsf x,u)=\{\mathsf y\in \mathsf {X}:d(\mathsf y,\mathsf x)\ll u\},~\mbox{for}~ x\in \mathsf {X} ~\mbox{and}~ u\in \mathsf{P}^{\circ}.$$
\end{definition}
\begin{example}
Let $\mathsf{X}=\mathbb{R},\mathsf{E}=\mathbb{R},\mathsf{P}=[0,\infty).$ Then $\mathsf{P}^{\circ}=(0,\infty).$ Define $d:\mathsf{X}\times\mathsf{X} \to \mathsf{E}$ by 
\begin{equation*}
d(\mathsf x,\mathsf y)=
\left\{
\begin{aligned}
\mathsf y-\mathsf x~~~~~if~~\mathsf y\geq \mathsf x;\\
1~~~~~otherwise.
\end{aligned}
\right.
\end{equation*}
Then $(\mathsf{X},d)$ is a quasi cone metric space. The forward open ball is $\mathsf{B}_{f}(\mathsf x,\mathsf{u})=[\mathsf x,\mathsf x+u),$ where $1\ll u$ and the backward open ball is 
$\mathsf{B}_{b}(\mathsf x,\mathsf{u})= (\mathsf x-u,\mathsf x],$ where $1\ll u.$
\end{example}
\begin{definition}
The forward open balls 
$$\mathsf B_{f}(\mathsf x,u)=\{\mathsf y\in \mathsf {X}:d(\mathsf x,\mathsf y)\ll u\},~\mbox{for}~ \mathsf x\in \mathsf X ~\mbox{and}~u\in \mathsf{P}^{\circ}.$$
generate the forward topology $\mathsf{T}_{f},$ where $\mathsf{T}_{f}$ is induced by  $d.$
Similarly, the backward open balls 
$$\mathsf B_{f}(\mathsf x,u)=\{\mathsf y\in \mathsf {X}:d(\mathsf y,\mathsf x)\ll u\},~\mbox{for}~ \mathsf x\in \mathsf X ~\mbox{and}~u\in \mathsf{P}^{\circ}.$$
generate the backward topology $\mathsf{T}_{b},$ where $\mathsf{T}_{b}$ is also induced by  $d.$
\end{definition}
 A set $H \subset \mathsf{T}_{f}$ is said to be forward closed if $H^{c}\in \mathsf{T}_{f}$ and a set $H \subset \mathsf{T}_{b}$ is said to be backward closed if $H^{c}\in \mathsf{T}_{b}.$

\begin{remark}
Suppose $(\mathsf {X},d)$ is a quasi cone metric space. Then $(\mathsf {X}, \mathsf{T}_{f})$~\mbox{and}~$(\mathsf {X},\mathsf{T}_{b})$ are $T1$-spaces.
\end{remark}

\begin{remark}
Let $\mathbb{B}_{f}=\{ \mathsf{B}_{f}(\mathsf x, u):\mathsf x\in \mathsf{X},u\in \mathsf{P}^{\circ}\}$ denote the basis for forward topology $\mathsf{T}_{f}$ on quasi cone metric space $(\mathsf{X},d)$  and $\mathbb{B}_{b}=\{ \mathsf{B}_{b}(\mathsf x, u):\mathsf x \in \mathsf{X},u \in \mathsf{P}^{\circ}\}$ denote the basis for backward topology $\mathsf{T}_{b}$ on quasi cone metric space $(\mathsf {X},d).$ 
\end{remark}

\begin{theorem}\label{thm:Th1}
A quasi cone metric space  $(\mathsf {X},d)$ with forward topology $\mathsf{T}_{f}$ is $T_{2}$ -space if forward convergence in $(\mathsf {X},d)$ implies backward convergence.
\end{theorem}
\begin{proof}
Let $\mathsf x$ and $\mathsf y$ be two distinct points in $\mathsf {X}.$ Suppose that for all $u\in \mathsf{P}^{\circ}, \mathsf B_{f}(\mathsf x,u) \cap \mathsf B_{f}(\mathsf y,u) \neq \phi.$  Then there exists $z_{u} \in  \mathsf B_{f}(\mathsf x,u) \cap \mathsf B_{f}(\mathsf y,u).$ Now, if $\mathsf z_{u} \in  \mathsf B_{f}(\mathsf x,u) \cap \mathsf B_{f}(\mathsf y,u),$ then $d(\mathsf x,z_{u})\ll u$ and $d(\mathsf y,z_{u}) \ll u.$ Therefore, $\mathsf z_{u}$ is forward convergent to $x$ and $\mathsf z_{u}$ is forward convergent to $\mathsf y.$ Since forward convergence implies backward convergence. Then $\mathsf z_ {u}$ is backward convergent to $\mathsf y.$ Hence, we have $d(\mathsf x,\mathsf y)\leq d(\mathsf x,\mathsf z_{u}) +d(\mathsf z_{u}, \mathsf y)\ll u+u = 2u,$ for all $u\in \mathsf{P}^{\circ},\mathsf{P}^{\circ}\neq \phi.$ Therefore, for each $n\in \mathbb{N}$ we can choose $v\in \mathsf{P}^{\circ}.$ Then $\frac {v} {n}-d(\mathsf x,\mathsf y) \in  \mathsf{P}.$ If we take $n\to \infty,$ then $-d(\mathsf x,\mathsf y)\in \mathsf{P}$ as a cone is closed. Hence $d(\mathsf x,\mathsf y)=0$ which implies that $\mathsf x=\mathsf y,$ a contradiction. Hence the proof.
\end{proof}
\begin{theorem}
A quasi cone metric space  $(\mathsf {X},d)$ with backward topology $\mathsf{T}_{b}$ is $T_{2}$-space if backward convergence in $(\mathsf {X},d)$ implies forward convergence.
\end{theorem}
\begin{proof} Omitted as follows from Theorem \ref{thm:Th1} .
\end{proof}

\begin{proposition}\label{prop:P}
In a quasi cone metric space $(\mathsf {X},d), \mathsf{T}_{f}\subseteq \mathsf{T}_{b}$ iff for every $\mathsf x\in \mathsf{X}$ and every $\mathsf{B}_{b}\in \mathbb{B}_{b}$
that contains $\mathsf x,$ there exists  $\mathsf{B}_{f}\in \mathbb{B}_{f}$ such that $\mathsf x\in \mathsf{B}_{f} \subseteq \mathsf{B}_{b}.$
\end{proposition}
 
\begin{proof}
Straightforward hence omitted.
\end{proof}

\begin{lemma}
Suppose $(\mathsf{X},d)$ is a quasi cone metric space in which forward convergence implies backward convergence. Then $\mathsf{T}_{f}\subseteq \mathsf{T_{b}}.$
\end{lemma}
\begin{proof}
Let $u$ be an arbitrary point of $\mathsf{P}^{\circ}$ and $\mathsf x\in \mathsf{B}_{f}(\mathsf y,u).$ Then $d(\mathsf x,\mathsf y)\ll u.$ Therefore, $\mathsf x$ is forward convergent to $\mathsf y.$ Also, 
$x$ is backward convergent to $\mathsf y$ as forward convergence implies backward convergence. Hence $d(\mathsf y,\mathsf x)\ll u$ implies that $\mathsf x\in \mathsf{B}_{b}(\mathsf y,\mathsf u).$ By using Proposition \ref{prop:P}, we obtain the desired result.

\begin{remark}
If a quasi cone metric space $(\mathsf{X},d)$ is compact and Hausdorff, then $\mathsf{T}_{f}$ and $\mathsf{T}_{b}$ are equivalent or $\mathsf{T}_{f}$ and $\mathsf{T}_{b}$ are incomparable.
\end{remark}       
 
\end{proof}
\begin{theorem}\label{thm:TH2}
Every quasi cone metric space $(\mathsf {X},d)$ with forward topology $\mathsf{T}_{f}$ is first countable.
\end{theorem}
\begin{proof}
Let $u$ be an arbitrary point of $\mathsf{P}^{\circ}.$ Consider the family $\mathbb{B}_{f}=\{ \mathsf{B_{f}(\mathsf x,\frac{1}{n} u)}:n\in \mathbb{N},\mathsf x\in \mathsf {X}\}.$ The family $\mathbb{B}_{f}$ is a countable collection of neighborhoods. Suppose that  $\mathsf{B}_{f}(\mathsf x,v)$ is a neighborhood. Then there exists $n\in \mathbb{N}$ such that $\frac{1}{n} u \ll v.$ Hence $\mathsf{B}_{f}(\mathsf x,\frac{1}{n} u)\subset \mathsf{B}_{f}(\mathsf x,v).$ This completes the proof.   
\end{proof}
\begin{theorem}
Every quasi cone metric space $(\mathsf {X},d)$ with backward topology $\mathsf{T}_{b}$ is first countable.
\end{theorem}
\begin{proof} Omitted as follows from Theorem \ref{thm:TH2}.
\end{proof}

\begin{definition}
$\mathsf{A}\subset \mathsf{X}$ is forward (resp. backward) bounded if there exist $\mathsf x\in \mathsf{X}$ and $u\gg 0$ such that  $\mathsf{A}\subset \mathsf{B}_{f}(\mathsf x,u)$ (resp. $\mathsf{A}\subset \mathsf{B}_{b}(\mathsf x,u))$. 
  
\end{definition}
\begin{definition}
A set $\mathsf{A}\subset \mathsf{X}$ is forward totally bounded(resp. backward totally bounded) if there exist $\mathsf x_{1},\mathsf x_{2},\ldots,\mathsf x_{n} \in \mathsf{A}$ such that $\mathsf{A}\subset \displaystyle\bigcup_{i=1}^{n} \mathsf{B}_{f}(\mathsf x_{i},u),u\in \mathsf{P}^{\circ}$ (resp. $\mathsf{A}\subset \displaystyle\bigcup_{i=1}^{n} \mathsf{B}_{b}(\mathsf x_{i},u),u\in \mathsf{P}^{\circ}).$ 
\end{definition}

\begin{definition}
If each open cover of $\mathsf{A}$ in the forward (resp. backward) topology admits a finite subcover, then the set $\mathsf{A} \subset \mathsf{X}$ is forward (resp. backward) compact.
\end{definition}

\begin{definition}  
If $\bar{\mathsf{A}}$ is forward (backward) compact, then the set $\mathsf{A} \subset \mathsf{X}$ is forward (resp. backward) relatively compact, where $\bar {\mathsf{A}}$ represents the closure of set $A$ in the forward (resp. backward) topology.
\end{definition}
 
\begin{proposition}\label{prop:1} 
A quasi cone metric space  $(\mathsf{X},d)$ is forward compact if $\mathsf{X}$ is forward sequentially compact and forward totally bounded.
\end{proposition}
\begin{proof}
Assume that in the forward topology, $\mathcal{S}$ is an open cover for $\mathsf{X}$ in forward topology $\mathsf{T_{f}}$ on the contrary that any $w\gg 0$ there exists $v< w$ such that no element of $\mathcal{S}$ contains the ball with radius $v.$ Particularly, no element of $\mathcal{S}$ contains the ball, $\mathsf{B}_{n}=\mathsf{B}_{f}(\mathsf x_{n},r),$ where $r<\frac{1}{n},n\in \mathbb{N}.$ Since $\mathsf{X}$ is forward sequentially compact. Then there is a point $b$(say) such that the subsequence $\{\mathsf x_{n_{i}}\}$ forward converges to $b.$ Clearly, there is $\mathsf{S} \subset \mathcal{S}$ such that $b\in \mathsf{S}.$ Now, let $u \gg 0$ such that $\mathsf{B}_{f}(b,u)\subset \mathsf{S} $ as  $\mathsf{S}$ is open. Also, if we take $i$ sufficiently large such that $\frac{1}{n_{i}} < \frac{u}{2},$ then $\mathsf{B}_{n_{i}} \subset \mathsf{B}_{f}(b,u).$ If we also take $i$ sufficiently large such that $d(b,\mathsf x_{n_{i}}) \ll \frac{u}{2},$ then $\mathsf{B}_{n_{i}}\subset \mathsf{B}_{f}(b,u).$ Therefore, $\mathsf{S}$ contains $\mathsf{B}_{n_{i}}$ which is a contradiction. We conclude that there is an element of $\mathcal{S}$ containing the ball with radius $v,$ where  $v<w$ for any $w\gg 0.$ Choose $u=\frac{w}{2}.$ Since $\mathsf{X}$ is totally bounded we can cover $\mathsf{X}$ by forward open balls. Each of these balls are contained in an element of   $\mathcal{S}$ as these balls have radius $u<w.$ Hence  we obtain a finite subcollection of $\mathcal{S}$ that covers $\mathsf{X}$ by choosing one element of $\mathcal{S}$ corresponding to each of these forward $u$-balls. This proves the theorem.
\end{proof}

\begin{proposition}\label{prop:2}
A quasi cone metric space  $(\mathsf{X},d)$ is forward totally bounded if $\mathsf{X}$ is forward sequentially compact and forward convergence implies backward convergence.
\end{proposition}

\begin{proof}
Suppose on the contrary that $\mathsf{X}$ is not forward totally bounded. Then there exists $u\gg 0$ such that finitely many forward $u$-balls cannot cover $\mathsf{X}.$ Let $\mathsf x_{1} \in \mathsf{X}.$ Since $\mathsf{B}_{f}(\mathsf x_{1},u)\neq \mathsf{X}.$ Then we can choose $\mathsf \mathsf x_{2}$ in complement of $\mathsf{B}_{f}(\mathsf x_{1},u)$ in $\mathsf{X}.$ Proceeding inductively, we can choose $x_{n+1}$ in complement of $\mathsf{B}_{f}(\mathsf x_{1},u) \cup \ldots \cup  \mathsf{B}_{f}(\mathsf x_{n},u),$ in $\mathsf{X}.$ Then $d(\mathsf x_{i},\mathsf x_{n+1})\geq u$ for $i=1,2,\ldots n.$ Since  $\mathsf{X}$ is forward sequentially compact, the sequence $\{\mathsf x_{n}\}$ has a forward convergent subsequence. Then by hypothesis this subsequence is backward convergent. Hence the subsequence is forward  Cauchy, which is a contradiction. Hence the proof.
\end{proof}

\begin{lemma}\label{lemma:1}
A  set $\mathsf{K}$ in a quasi cone metric space $(\mathsf{X},d)$ is forward sequentially compact if $\mathsf{K}$ is forward compact.
\end{lemma}
\begin{proof}
Suppose that $\mathsf{K}$ is not forward sequentially compact and $\{\mathsf x_{n}\}$ be a sequence in $\mathsf{K}.$ Then there exists a non forward convergent subsequence of $\{\mathsf {x_{n}}\}.$ Let $\mathsf x\in \mathsf{X}.$ Clearly, only finite number of terms of sequence $\{\mathsf x_{n}\}$ are contained in $\mathsf{B}_{f}(\mathsf x,u_{\mathsf x}).$ Also the set $\mathcal{B}=\{\mathsf{B}_{f}(\mathsf x,u_{\mathsf x}):\mathsf x\in \mathsf{X},u_{\mathsf x}\gg0 \}$ forms an open covering of $\mathsf{X}.$ Then for $n\in \mathbb{N},$ $\mathcal{B}_{1}=\{\mathsf{B}_{f}(\mathsf x_{i},u_{\mathsf x_{i}}):\mathsf x\in \mathsf{X},i=1,2,\ldots,n\}$ is a finite subcover of $\mathsf{X},$ a contradiction as only finite number of terms of sequence $\{\mathsf x_{n}\}$ are contained in $\mathsf{B}_{f}(\mathsf x_{i},u_{\mathsf x_{i}}).$ Hence the proof.      
\end{proof}

\begin{lemma}\cite{YHC}\label{lemma:2}
A quasi cone metric space $(\mathsf{X},d)$ is forward complete if every forward Cauchy sequence has a forward convergent subsequence.
\end{lemma}

\begin{proposition}\label{prop:3}
A quasi cone metric space $(\mathsf{X},d),$ is forward compact if and only if  $\mathsf{X}$ is forward totally bounded and forward complete.
\end{proposition}

\begin{proof}
Suppose $\mathsf{X}$ is forward compact. Then using Lemma {\ref{lemma:1}}, $\mathsf{X}$ is forward sequentially compact. This implies every forward Cauchy sequence has a subsequence that converges forward. By Lemma {\ref{lemma:2}}, $\mathsf{X}$ is forward complete. Clearly, $\mathsf{X}$ is forward totally bounded as every open covering of $\mathsf{X}$ by forward $u$-balls admits a finite subcover.\\
Conversely, let $\mathsf{X}$ be forward complete and forward totally bounded. To prove that $\mathsf{X}$ is forward compact we shall prove that $\mathsf{X}$ is forward sequentially compact. Consider a sequence $\{\mathsf x_{n}\}$ in $\mathsf{X}.$ Let $\mathcal{B}=\{B_{f}(\mathsf x,1):\mathsf x\in \mathsf{X}\}$ be an open covering of $\mathsf{X}.$ Then for infintely many values of $n, \{\mathsf x_{n}\}$ is contained in one of these balls, say $\mathsf{B}_{1}.$ Let $\mathcal{J}_{1}=\{n\in \mathbb{N}:\mathsf x_{n}\in \mathsf{B}_{1}\}.$ Then $\mathsf{B_{1}}$ can be covered by finite number of balls of radius $\frac{1}{2}$ as $\mathsf{B_{1}}$ is forward totally bounded. For infinitely many values of $n$ in $\mathcal{J}_{1},$ $\{\mathsf x_{n}\}$ is contained in one of the balls, $\mathsf{B_{2}}$(say). Proceeding in this manner we obtain a sequence of balls  $\mathsf{B_{1}}\supset \mathsf{B_{2}} \supset \mathsf{B_{3}}\ldots$ and corresponding sets $\mathcal{J}_{2}=\{n:n\in \mathcal{J}_{1},\mathsf x_{n}\in \mathsf{B_{2}}\},~ \mathcal{J}_{3}=\{n:n\in \mathcal{J}_{2},x_{n}\in \mathsf{B}_{3}\}.$ In general $\mathcal{J}_{m+1}=\{n:n\in \mathcal{J}_{m},\mathsf x_{n}\in \mathsf{B_{m+1}}\}.$ Let $n_{1}\in \mathcal{J}_{1}.$ For $n_{m},$ let $n_{m+1}\in \mathcal{J}_{m+1}$ such that $n_{m}< n_{m+1}.$ Then for each $k\in \mathbb{N},\mathsf  x_{n_{k}}\in \mathsf{B_{k}}.$ Consider the ball $\mathsf{B_{k}}=\mathsf{B}(\mathsf y_{k},\frac{1}{k}).$ Since $d(\mathsf y_{k},\mathsf y_{k+1})\ll \frac{1}{k}$ tends to $0$ as $k$ tends to $\infty.$ Then, if $k\leq m$ we have $d(\mathsf y_{k}, \mathsf y_{m})$ tends to $0$ as $k$ tends to $\infty.$ From this we deduce that $\{\mathsf x_{m}\}$ is a forward Cauchy sequence. Now, since $\mathsf{X}$ is forward complete.Then there exists $\mathsf y$ in $\mathsf{X}$ such that $\{y_{m}\}$ is forward convergent to $y.$ Lastly, $d(y,x_{n_{k}})\leq d(y,y_{k})+d(y_{k},\mathsf x_{n_{k}}).$ Then $d(\mathsf y,\mathsf x_{n_{k}})\to 0$ as $k\to \infty.$ Therefore, $\{\mathsf x_{n_{k}}\}$ is a forward convergent subsequence of $\{\mathsf x_{n}\}$ which proves that $X$ is forward sequentially compact.
\end{proof}

\begin{proposition}
Suppose that in a quasi cone metric space $(\mathsf{X},d)$ forward convergence implies backward convergence. Then every forward compact set is backward totally bounded, so it is backward bounded.
\end{proposition}

 \begin{proof}
Suppose that $\mathsf{K}$ is a non-empty forward compact set in  $\mathsf{X}$ and for $y\in \mathsf{K},u \gg 0,$ finitely many backward open balls $\mathsf{B}_{b}(\mathsf y,u)$ cannot cover $\mathsf{K}.$ Let $\mathsf y_{1}\in \mathsf{K}.$ Since $\mathsf{K}$ is not contained in $\mathsf{B}_{b}(\mathsf y_{1},u).$ Then there exist $\mathsf y_{2}$ in complement of $\mathsf{B}_{b}(\mathsf y_{1},u)$ in $\mathsf{K}.$ By induction, we see that $\mathsf{K}$ is not contained in $\displaystyle \bigcup_{j=1}^{n} \mathsf{B}_{b}(\mathsf y_{j},u).$ Then we can choose $\mathsf y_{m+1}$ in complement of $\displaystyle\bigcup_{j=1}^{n} \mathsf{B}_{b}(\mathsf y_{j},u)$ in  $\mathsf{K}.$ Hence there exists a sequence $\{\mathsf y_{m}\}$ in $\mathsf{K}$ so that 
\begin{equation}\label{eq:1}
d(\mathsf y_{m+1}, \mathsf y_{j})\geq u, \forall m\in \mathbb{N},\forall j=1,2,\ldots,m.
\end{equation}
Then using Lemma \ref{lemma:1}, we obtain a subsequence  $\{\mathsf y_{m_{k}}\} $ of $\mathsf y_{m}$ such that $\{\mathsf y_{m_{k}}\}$ is forward convergent to $\mathsf y$ in $\mathsf{K}.$ From this we deduce that $\{\mathsf y_{m_{l}}\}$ is backward convergent to $\mathsf y.$ Hence there is a natural number $l_{0}$ such that $d(\mathsf y, \mathsf y_{m_{l}})\ll \frac{u}{2}$ and $d(\mathsf y_{m_{l}},\mathsf y)\ll \frac{u}{2}$ for all $l\geq l_{0}.$ As a result, using (\ref{eq:1}),we obtain
\begin{equation*} 
u\leq d(\mathsf y_{n_{l}+1},\mathsf y_{n_{l}})\leq d(\mathsf y_{n_{l}+1},\mathsf y) + d(\mathsf y,\mathsf y_{n_{l}})\ll u 
\end{equation*}
which does not hold. Hence the proof.
\end{proof}
\begin{definition}
Let $(\mathsf{X},d_{\mathsf{X}})$ and $(\mathsf{Y},d_{\mathsf{Y}})$ be two quasi cone metric spaces. A set of functions $\mathcal{S}$ from $\mathsf{X}$ to $\mathsf{Y}$ is said to be forward(resp. backward) equicontinuous if for every $u_{1}\gg 0$ and for every $\mathsf x\in \mathsf{X},$ there exists $u_{2}\gg 0$ such that for every $f\in \mathsf{S}$ and $\mathsf y\in \mathsf{Y}$ we have $d_{\mathsf{Y}}(f(\mathsf x),f(\mathsf y))\ll u_{2}$ (resp. $d_{\mathsf{Y}}(f(\mathsf y),f(\mathsf x))\ll u_{2})$ whenever $d_{\mathsf{X}}(\mathsf x,\mathsf y)\ll u_{1}.$
\end{definition}
\begin{lemma}\cite{YHC}\label{lemma:3}
Suppose that a quasi cone metric space $(\mathsf{X},d)$ is forward sequentially compact. Then backward convergence implies forward convergence.
\end{lemma}

\begin{lemma}\label{lemma:4}
Suppose $(\mathsf{X},d)$ be a quasi cone metric space that is forward compact and forward convergence in $\mathsf{X}$ implies backward convergence. Then the following assertions hold:
\begin{itemize}
\item [(a)] A forward equicontinuous subset, $\mathcal{S}$ of $\mathsf{C}(\mathsf{X},\mathsf{Y})$ is also backward equicontinuous.
\item [(b)] A uniformly backward convergent sequence, $\{f_{n}\}$ in $\mathsf{Y}^{\mathsf{X}}$ is also uniformly forward convergent.
\end{itemize}
\end{lemma}
\begin{proof}
Clearly, using Lemma \ref{lemma:1}, we see that $\mathsf{Y}$ is forward sequentially compact. Then $\mathsf{Y}$ is backward sequentially  compact as forward convergence implies backward convergence. Now, $\mathsf{Y}$ is backward sequentially compact. Then using Proposition \ref{prop:2}, we see that $\mathsf{Y}$ is backward totally bounded. Since $\mathsf{Y}$ is backward sequentially compact and backward totally bounded then using Proposition \ref{prop:1}, we can conclude that $\mathsf{Y}$ is backward compact. We see that both forward and backward topology are equivalent when the space is both forward and backward compact. As a result we conclude that forward equicontinuity implies backward equicontinuity. The similar result holds in case of uniform convergence.
\end{proof}
\begin{proposition}\label{prop:4}
Let $(\mathsf{X},d_{\mathsf{X}})$ and $(\mathsf{Y},d_{\mathsf{Y}})$ be two quasi cone metric spaces with compact forward topologies and forward convergence in  
$(\mathsf{Y},d_{\mathsf{Y}})$ implies backward convergence. A set $\mathcal{S}\subset \mathsf{C}(\mathsf{X},\mathsf{Y})$ is forward totally bounded in the uniform metric $\bar{\rho}$ relative to $d_{\mathsf{Y}}$ if $\mathcal{S}$ is forward equicontinuous.
\end{proposition}
\begin{proof}
Suppose that $\mathcal{S}$ is forward equicontinuous. Then using Lemma \ref{lemma:4}, $\mathcal{S}$ is also backward equicontinuous. To prove that $\mathcal{S}$ is forward totally bounded, we shall show that for $0 \ll u\ll 1,$ the finite number of forward $u$-balls form an open cover of $\mathcal{S}$ in the metric $\bar{\rho}.$ Take $w=\frac{u}{3}.$ Choose a point $b$ in $\mathsf{X}.$ Then there exists $w_{b}\gg 0$ such that $d_{\mathsf{Y}}(f(b),f(\mathsf x))\ll w$ whenever $d_{\mathsf{X}}(b, \mathsf x)\ll w_{b},$ for each $\mathsf x\in \mathsf{X}$ and $f\in \mathcal{S}.$ Suppose the collection $\mathcal{B}=\{\mathsf{B}_{f}(b,w_{b}):b=b_{1},b_{2},\ldots,b_{l}\}$ is an open covering of $\mathsf{X}$ since $\mathsf{X}$ is forward compact. The finite number of open sets $\mathsf{V_{1}},\mathsf{V_{1}},\ldots,\mathsf{V_{k}}$ can cover $\mathsf{Y}$ as $\mathsf{Y}$ is also forward compact. Note that diameter of $\mathsf{V_{i}}$ denoted by dia$(V_{i})=\displaystyle\max_{\mathsf x, \mathsf y\in \mathsf{V}}d(\mathsf x, \mathsf y)\ll w,i=1,2,\ldots,k.$ and $\mathsf{B_{f}}(b,u)\subset \mathsf{B_{f}}(b,2u).$ Consider the set $\mathcal{J}=\{\beta|\beta:\{1,\ldots,l\}\to\{1,\ldots,k\}\}.$Choose a function $\beta$ in $\mathcal{S}.$ Then there exists $f$ in $\mathcal{S}$ such that $f(b_{i})\in \mathsf{V}_{\beta_{i}}$ where $i=1,2,\ldots,l.$ Take one function and mark it as $f_{\beta}.$ Next consider the sequence $\{f_{\beta}\}$ and the set $\mathcal{J}_{1}=\{\beta \in \mathcal{J}:f(b_{i})\in \mathsf{V}_{\beta_{i}}\}$ Then $\mathcal{J}_{1}$ indexes the sequence $\{f_{\beta}\}.$ Hence $\mathcal{J}_{1}$ is contained in $\mathcal{J}.$ As a result, $\{f_{\beta}\}$ is finite.\\
Claim: $\mathcal{B}=\{(\mathsf{B}_{f})_{\bar{\rho}}(f_{\beta},u):\beta\in \mathcal{J}_{1}\}$ is an open covering for $\mathcal{S}.$\\ 
Suppose $f\in \mathcal{S}$ and take $\beta_{i},i=1,2,\ldots,k$ so that $f(\beta_{i})\in \mathsf{V}_{\beta_{i}}.$ This implies that $\beta \in \mathcal{J}_{1}.$ Next we shall prove that $(\mathsf{B}_{f})_{\bar{\rho}}(f_{\beta},u)$ contains $f.$ Choose $i$ such that $\mathsf x\in \mathsf{B}_{f}(b_{i},w_{i}),$ where $\mathsf x \in \mathsf{X}.$ Then $d(\mathsf x,b_{i})\ll w.$ Since $\mathcal{S}$ is equicontinuous. Then we have $d_{\mathsf{Y}}(f_{\beta}(\mathsf x),f_{\beta}(b_{i}))\ll w.$ Now, $\mathsf{V}_{\beta_{i}}$ contains $f(b_{i})$ and $f_{\beta}(b_{i}).$ Then $d_{\mathsf{Y}}(f_{\beta}(b_{i}),f(b_{i}))\ll w.$ Also, by equicontinuity of $f$ we see that $d_{\mathsf{Y}}(f({b_{i}}),f({\mathsf x})\ll w.$ Now,
\begin{eqnarray*}
d_{\mathsf{Y}}(f_{\beta}(\mathsf x),f(\mathsf x))
&\leq& d_{\mathsf{Y}}(f_{\beta}(\mathsf x),f_{\beta}(b_{i}))+d_{\mathsf{Y}}(f_{\beta}(b_{i}),f(b_{i}))+d_{\mathsf{Y}}(f(b_{i}),f(\mathsf x))\\
&\ll&\frac{u}{3}+\frac{u}{3}+\frac{u}{3}\\
&=&u\\
&\ll&1.
\end{eqnarray*}
For each $\mathsf x\in \mathsf{X}$ the above inequality holds. Therefore,
\begin{equation*} 
\bar{\rho}(f_{\beta},f)=\max\{\bar{d}(f_{\beta}(\mathsf x),f(\mathsf x))\}\ll u.
\end{equation*}
Hence the claim.
\end{proof}
\begin{lemma}\cite{YHC}\label{lemma:5}
Let $(\mathsf{X},d_{\mathsf{X}})$ and $(\mathsf{Y},d_{\mathsf{Y}})$ be two quasi cone metric spaces. Then $\mathsf{Y}^{\mathsf{X}}$ is forward complete if $\mathsf{Y}$ is forward complete.
\end{lemma}
\begin{lemma}
The space $\mathsf{Y}^{\mathsf{X}}$ is complete in the uniform metric $\bar{\rho}$ relative to $d,$ if in the quasi cone metric space $(\mathsf{Y},d)$ forward convergence implies backward convergence and $\mathsf{Y}$ is forward compact.
\end{lemma}
\begin{proof}
We can easily see that $(\mathsf{Y},d)$ is forward complete as $\mathsf{Y}$ is forward compact, using Proposition \ref {prop:3}. Then $(\mathsf{Y},\bar{d})$ is also forward complete. We shall now show that $(\mathsf{Y},\bar{d})$ is backward complete. Suppose that $\{f_{k}(\mathsf x)\}$ is a forward Cauchy sequence in $\mathsf{Y}^{\mathsf{X}}$ corresponding to $\bar{\rho}.$ Then $\{f_{k}(\mathsf x)\}$ is forward Cauchy in $(\mathsf{Y},\bar{d}).$ Then there exists $\mathsf y'$  such that  $\{f_{k}(\mathsf x)\}$ is forward and backward convergent to $\mathsf y'.$ Define a function $f:\mathsf{X}\to \mathsf{Y}$ by $f(\mathsf x)=\mathsf y'.$ We shall first show that $\{f_{k}(\mathsf x)\}$ is backward convergent to $f$ in order to show that $\{f_{k}(\mathsf x)\}$ is forward convergent to $f$ relative to $\bar{\rho}.$ Choose $u\gg 0$ and $\mathsf{N}\in \mathbb{N}$ such that $\bar{\rho}(f_{k},f_{l})\ll \frac{u}{2}$ for $k\geq l\geq \mathsf{N}.$ Particularly, for each $\mathsf x\in \mathsf{X}$ we have $\bar{d}(f_{k}(\mathsf x),f_{l}(\mathsf x))\ll \frac{u}{2}.$ Let us fix $\mathsf x$ and $k.$ For $l$ sufficiently large, we see that 
\begin{equation*}
\bar{d}(f_{k}(\mathsf x),f(\mathsf x))\leq \bar{d}(f_{k}(\mathsf x),f_{l}(\mathsf x))+\bar{d}(f_{l}(\mathsf x),f(\mathsf x))\ll u,\forall \mathsf x\in \mathsf{X},k\geq\mathsf{N}.
\end{equation*}
Therefore, we deduce that $\bar{\rho}(f_{k},f)\ll u,k\geq \mathsf{N}.$ This shows that $\{f_{k}(\mathsf x)\}$ is also uniformly backward convergent. Hence using Lemma {\ref{lemma:4}}, we deduce that $\{f_{k}(\mathsf x)\}$ is also uniformly forward convergent. Hence the proof.
\end{proof}
\begin{lemma}\cite{YHC}\label{lemma:6}
Let $(\mathsf{X},d_{\mathsf{X}})$ and $(\mathsf{Y},d_{\mathsf{Y}})$ be two quasi cone metric spaces and forward convergence in $\mathsf{Y}$ is equivalent to backward convergence. Then the set of all $ff$-continuous functions, $\mathsf{C}(\mathsf{X},\mathsf{Y})$ is a closed subset of all continuous functions. 
\end{lemma}
\begin{proposition}\label{prop:5}
 Let $(\mathsf{X},d_{\mathsf{X}})$ and $(\mathsf{Y},d_{\mathsf{Y}})$ be two quasi cone metric spaces. Then the space $\mathsf{C}(\mathsf{Y},\mathsf{X})$ is complete in the uniform metric $\bar{\rho}$ relative to $d_{Y},$ if in the quasi cone metric space $(\mathsf{Y},d)$ forward convergence implies backward convergence and $\mathsf{Y}$ is compact.
\end{proposition} 
\begin{proof}
Take a sequence $\{f_{k}\}$ in $\mathsf{C}(\mathsf{X},\mathsf{Y})$ that is uniformly convergent $f\in {\mathsf{Y}}^{\mathsf{X}}$ corresponding to $\bar{\rho}.$ It is assumed that forward convergence implies backward convergence. Now, using Lemma {\ref{lemma:1}}, we see that $\mathsf{Y}$ is forward sequentially compact as $\mathsf{Y}$ is forward compact. So, by Lemma {\ref{lemma:3}}, backward convergence implies forward convergence. Thus, both forward and backward convergence are equivalent.Then  using Lemma {\ref{lemma:6}}, we infer that $\mathsf{C}(\mathsf{X},\mathsf{Y})$ is closed. Hence $f\in \mathsf{C}(\mathsf{X},\mathsf{Y}).$ From this we deduce that $\mathsf{C}(\mathsf{X},\mathsf{Y})$ is complete in the uniform metric $\bar{\rho}$ corresponding to $d.$ Hence the proof.
\end{proof}
\begin{definition}
Let $(\mathsf{X},d_{\mathsf{X}})$ and $(\mathsf{Y},d_{\mathsf{Y}})$ be two quasi cone metric spaces. A set $\mathcal{S}\subset \mathsf{C}(\mathsf{X},\mathsf{Y})$ is said to be forward pointwise bounded with respect to $d_{\mathsf{Y}}$ if for each $b\in \mathsf{X},$ the set $\mathcal{S}_{b}=\{f({b}):f\in \mathcal{S}\subseteq \mathsf{Y}\}$ is bounded in $d_{\mathsf{Y}}.$
\end{definition}  

Finally, we shall prove Arzela Ascoli Theorem in quasi cone metric space.

\begin{theorem}(Arzela Ascoli Theorem)
Let $(\mathsf{X},d_{\mathsf{X}})$ and $(\mathsf{Y},d_{\mathsf{Y}})$ be two quasi cone metric spaces where $\mathsf{X}$ is forward compact. In $\mathsf{Y}$ forward convergence implies backward convergence, bounded and closed sets are forward compact in $\mathsf{Y}.$ The space $\mathsf{C}(\mathsf{X},\mathsf{Y})$ is endowed with $\bar{\rho}$ where $\bar{\rho}$ is the uniform metric corresponding to $d_{\mathsf{Y}}.$ If a set $\mathcal{S}\subset \mathsf{C}(\mathsf{X},\mathsf{Y})$ is forward equicontinuous and forward pointwise bounded under $d_{\mathsf{Y}},$ then $\mathcal{S}$ is forward relatively compact.
\end{theorem}
\begin{proof}
$\bar{\mathcal{S}}$ is forward equicontinuous: For $u\gg 0$ and $\zeta_{0}\in \mathsf{X},$ take $w\gg 0$ such that  $d_{\mathsf{X}}(\zeta_{0},\zeta)\ll w.$ By equicontinuity of $\mathcal{S},$ we have $d_{\mathsf{Y}}(f(\zeta_{0})f(\zeta))\ll \frac{u}{3},$ for each $f\in \mathcal{S}$ and $\zeta \in \mathsf{X}.$ Let $h\in \bar{\mathcal{S}}.$ Then there exists a sequence in $\mathcal{S}$ that is forward convergent to $h.$ Also, by hypothesis this sequence is also backward convergent to $h.$ Therefore, choose an element $f$ in $\mathcal{S}$ such that $\bar{\rho}(h,f)\ll \frac{u}{3}$ and  $\bar{\rho}(f,h)\ll \frac{u}{3}.$ Hence for $d_{\mathsf{X}}(\zeta_{0},\zeta)\ll w,$ we have
\begin{eqnarray*}
d_{\mathsf{Y}}(h(\zeta_{0}),h(\zeta))
&\leq& d_{\mathsf{Y}}(h(\zeta_{0}),f(\zeta_{0})) + d_{\mathsf{Y}}(f(\zeta_{0}),f(\zeta))+d_{\mathsf{Y}}(f(\zeta),h(\zeta))\\
&\ll& \frac{u}{3}+\frac{u}{3}+\frac{u}{3}\\
&=&u.
\end{eqnarray*}
Since $h$ is an arbitary element of $\bar{\mathcal{S}}$ we conclude that $\bar{\mathcal{S}}$ is equicontinuous.\\
$\bar{\mathcal{S}}$ is forward pointwise bounded: For $b \in \mathsf{X},$ let $\xi \in \mathsf{Y}$ such that $d_{\mathsf{Y}}(\xi, f(b)\ll K,\forall f\in \mathcal{S}.$ Using the same method as earlier, let $h\in \bar{\mathcal{S}},$ there exists a sequence in $\mathcal{S}$ that is forward and backward convergent to $h$ such that $\bar{\rho}(f,h)\ll1,f\in \mathcal{S}.$ Then we have $d_{\mathsf{Y}}(\xi,h(b))\leq d_{\mathsf{Y}}(\xi, f(b))+d_{\mathsf{Y}}(f(b),h(b))\ll K+1.$ Since $h\in \bar{\mathcal{S}}$ is  arbitrary. Then the set $\{d_{\mathsf{Y}}(\xi,h(b)):h\in \bar{\mathcal{S}}\}$ is bounded. This proves that $\mathcal{S}$ is forward pointwise bounded.\\
Next we shall show that for $h\in \bar{\mathcal{S}}, \cup h(\mathsf{X})\subseteq \mathsf{E}$ where $\mathsf{E}$ is a forward compact subspace of $\mathsf{Y}.$ For this, take $w_{b}\gg 0$ for each $b\in \mathsf{X},$ we have $d_{\mathsf{Y}}(h(b),h(\xi))\ll 1, h\in \bar{\mathcal{S}}$ whenever $d_{\mathsf{X}}(b,\xi)\ll w_{b}.$ Also the set $\mathcal{B}=\{\mathsf{B}_{f}(b_{i},w_{b_{i}}),i=1,2,\ldots,m\}$ can form an open covering of $\mathsf{X}$ as $\mathsf{X}$ is forward compact. Clearly, the union of sets $\{h(b_{i}):h\in \bar{\mathcal{S}}\}$ is forward bounded since the set $\{h(b_{i}):h\in \bar{\mathcal{S}}\}$ is forward bounded. Suppose that for $\xi \in \mathsf{Y},\mathsf{N}\in \mathbb{N}$, the union is contained in $\mathsf{B}_{f}(\xi, \mathsf{N}).$ \\
Let $x\in \mathsf{X}$ and $h$ be an arbitrary element of $\bar{\mathcal{S}}.$ Since $\mathsf{X}$ is forward compact. Then $x\in \mathsf{B}_{f}(b_{i},w_{i}),i=1,2,\ldots,n.$ Hence
\begin{eqnarray*}
 d_{\mathsf{Y}}(\xi,h(\zeta))&\leq&d_{\mathsf{Y}}({\xi},h(b_{i}))+ d_{\mathsf{Y}}(h(b_{i}),h(\zeta)))\\
                             &\ll&\mathsf{N}+1.
\end{eqnarray*}
Therefore, $h(\mathsf{X})$ is contained in $\mathsf{B}_{f}(\xi,\mathsf{N}+1).$ Let $\mathsf{E}$ denotes the closure of $\mathsf{B}_{f}(\xi,\mathsf{N}+1).$ 
Hence there exists a compact subset $\mathsf{E}$ of $\mathsf{Y}$ such that $\bar{\mathcal{S}}$ is contained in $\mathsf{C}(\mathsf{X},\mathsf{E}).$ Then using Proposition \ref{prop:4}, $\bar{\mathcal{S}}$ is forward totally bounded in the uniform metric $\bar{\rho}.$ From Proposition \ref{prop:5}, we deduce that $\mathsf{C}(\mathsf{X},\mathsf{E})$  is forward complete in the uniform metric $\bar{\rho}$ as $\mathsf{E}$ is forward compact. Since $\bar{\mathcal{S}}$ is a closed subspace of  $\mathsf{C}(\mathsf{X},\mathsf{E}).$  Then $\bar{\mathcal{S}}$ is also forward complete. Proposition \ref{prop:3} now reveals that $\bar{\mathcal{S}}$ is forward compact. Hence $\mathcal{S}$ is relatively compact. This proves the theorem.
\end{proof}

\end{document}